\documentclass[12pt,reqno]{amsart}
\usepackage{amscd,amsmath,amsthm,amssymb}
\usepackage{color}
\usepackage{stmaryrd}
\usepackage{graphicx}  
\usepackage{leftidx}
\usepackage{tikz}
\usepackage{tikz-cd}

\def\frk{\mathfrak}               

\def\mm{{\frk m}}

\def\G{{\mathcal G}}

\def\pd{\textup{proj}\phantom{.}\!\textup{dim}}

\def\fb{{\mathbf f}}
\def\opn#1#2{\def#1{\operatorname{#2}}} 
\opn\chara{char} \opn\length{\ell} \opn\pd{pd} \opn\rk{rk}
\opn\projdim{proj\,dim} \opn\injdim{inj\,dim} \opn\rank{rank}
\opn\depth{depth} \opn\grade{grade} \opn\height{height}
\opn\embdim{emb\,dim} \opn\codim{codim}

\opn\Tr{Tr} \opn\bigrank{big\,rank}
\opn\superheight{superheight}\opn\lcm{lcm}
\opn\trdeg{tr\,deg}
	\opn\reg{reg} \opn\lreg{lreg} \opn\ini{in} \opn\lpd{lpd}
	\opn\size{size} \opn\sdepth{sdepth}
	\opn\link{link}\opn\fdepth{fdepth}\opn\lex{lex}
	\opn\tr{tr}
	\opn\type{type}
	\opn\gap{gap}
	\opn\diam{diam}
	\opn\Mod{Mod}
	\opn\revlex{revlex}
	\opn\div{div} \opn\Div{Div} \opn\cl{cl} \opn\Cl{Cl}
	\opn\Spec{Spec} \opn\Supp{Supp} \opn\supp{supp} \opn\Sing{Sing}
	\opn\Ass{Ass} \opn\Min{Min}\opn\Mon{Mon}
	\opn\Ann{Ann} \opn\Rad{Rad} \opn\Soc{Soc}
	\opn\Im{Im} \opn\Ker{Ker} \opn\Coker{Coker} \opn\Am{Am}
	\opn\Hom{Hom} \opn\Tor{Tor} \opn\Ext{Ext} \opn\End{End}
	\opn\Aut{Aut} \opn\id{id}
	
	\opn\nat{nat}
	\opn\pff{pf}
	\opn\Pf{Pf} \opn\GL{GL} \opn\SL{SL} \opn\mod{mod} \opn\ord{ord}
	\opn\Gin{gin} \opn\Hilb{Hilb}\opn\sort{sort}
	\opn\PF{PF}\opn\Ap{Ap}
	\opn\dist{dist}
	\opn\aff{aff}
	\opn\relint{relint} \opn\st{st}
	\opn\lk{lk} \opn\cn{cn} \opn\core{core} \opn\vol{vol}  \opn\inp{inp} \opn\nilpot{nilpot}
	\opn\link{link} \opn\star{star}\opn\lex{lex}\opn\set{set}
	\opn\width{wd}
	\opn\Fr{F}
	\opn\QF{QF}
	\opn\G{G}
	\opn\type{type}\opn\res{res}
	\opn\conv{conv}
	\opn\sr{sr}
	\opn\gr{gr}
	
	\def\pot#1#2{#1[\kern-0.28ex[#2]\kern-0.28ex]}

	\opn\dirlim{\underrightarrow{\lim}}
	\opn\inivlim{\underleftarrow{\lim}}
	\let\tensor=\otimes
	\let\iso=\cong

	\let\to=\rightarrow
	\let\To=\longrightarrow
	\def\Implies{\ifmmode\Longrightarrow \else
		\unskip${}\Longrightarrow{}$\ignorespaces\fi}
	\def\implies{\ifmmode\Rightarrow \else
		\unskip${}\Rightarrow{}$\ignorespaces\fi}
	\def\iff{\ifmmode\Longleftrightarrow \else
		\unskip${}\Longleftrightarrow{}$\ignorespaces\fi}

	\let\:=\colon
	\newtheorem{Theorem}{Theorem}[section]
	\newtheorem{Lemma}[Theorem]{Lemma}
	\newtheorem{Corollary}[Theorem]{Corollary}
	\newtheorem{Proposition}[Theorem]{Proposition}
	\newtheorem{Remark}[Theorem]{Remark}
	
	\newtheorem{Example}[Theorem]{Example}
	\newtheorem{Examples}[Theorem]{Examples}

	\let\epsilon\varepsilon
	\let\kappa=\varkappa
	\def\qed{\ifhmode\textqed\fi
		\ifmmode\ifinner\hfill\quad\qedsymbol\else\dispqed\fi\fi}
	\def\textqed{\unskip\nobreak\penalty50
		\hskip2em\hbox{}\nobreak\hfill\qedsymbol
		\parfillskip=0pt \finalhyphendemerits=0}
	\def\dispqed{\rlap{\qquad\qedsymbol}}
	
	\opn\dis{dis}
	\def\pnt{{\raise0.5mm\hbox{\large\bf.}}}
	\def\lpnt{{\hbox{\large\bf.}}}
	\opn\Lex{Lex}
	\opn\Max{Max}
	\opn\Shad{Shad}
	\opn\astab{astab}

	\opn\Fitt{Fitt}
	
	\opn\v{v}
	
\begin{document}

	\title{Ideals and their Fitting ideals}
	\author{David Eisenbud, Antonino Ficarra, J\"urgen Herzog, Somayeh Moradi}

\address{Department of mathematics, University of California at Berkeley and the mathematical Science research institute, Berkeley, CA, 94720, USA 	
	}
	\email{de@msri.org}

	\address{Antonino Ficarra, Department of mathematics and computer sciences, physics and earth sciences, University of Messina, Viale Ferdinando Stagno d'Alcontres 31, 98166 Messina, Italy}
	\email{antficarra@unime.it}
	
	\address{J\"urgen Herzog, Fakult\"at f\"ur Mathematik, Universit\"at Duisburg-Essen, 45117 Essen, Germany} \email{juergen.herzog@uni-essen.de}
	
	\address{Somayeh Moradi, Department of Mathematics, Faculty of Science, Ilam University, P.O.Box 69315-516, Ilam, Iran}
	\email{so.moradi@ilam.ac.ir}
	
	\thanks{The fourth author is supported by the Alexander von Humboldt Foundation.
	}
	
	
	
	\maketitle
	
\begin{abstract}
For an ideal $I$ in a Noetherian ring $R$, the Fitting ideals $\Fitt_j(I)$ are studied. We discuss the question of when $\Fitt_j(I)=I$ or $\sqrt{\Fitt_j(I)}=\sqrt{I}$ for some $j$. A classical case is the Hilbert-Burch theorem when $j=1$ and $I$ is a perfect ideal of grade $2$ in a local ring. \end{abstract}

\section*{Introduction}
Fitting ideals are defined for finitely generated modules over a Noetherian ring. They convey  some information about the complexity of a module, as we will describe below.  Fitting ideals  have many applications in commutative algebra as well as in  number theory and arithmetic. In 1936, H. Fitting introduced these module invariants in \cite{Fitting36}.

Fitting ideals are defined as follows: for a  finitely generated $R$-module $M$ one chooses a  finite free presentation $G\stackrel{\varphi}{\To} F\to M\to 0$. Choosing bases of $F$ and $G$,  the $R$-module homomorphism $\varphi$  can be described by a matrix $A$. If $m$ is the rank of $F$, then  the $j$th Fitting ideal of $M$ is defined to be the ideal $\Fitt_j(M)=I_{m-j}(A)$, where for an integer $t$, $I_t(A)$ denotes the ideal of $t$-minors of $A$. Here  we use the convention that $I_t(A)=R$, for $t\leq 0$.  Fitting showed that  the ideals $\Fitt_j(M)$ are  invariants of $M$, that is,  they only depend on the module $M$ and do not depend  on the free presentation of $M$ nor on the choice of  the bases of the corresponding free modules. Obviously one has $\Fitt_0(M)\subseteq \Fitt_1(M)\subseteq\cdots \subseteq  \Fitt_m(M)=R$. Fitting ideals commute with the change of base rings, and in particular with localization.  In  other words, if $R\to S$ is a ring homomorphism, then $\Fitt_j(M\tensor_RS)=\Fitt_j(M) S$. An important property of  Fitting ideals is the following: the closed subset of $R$ defined by $\Fitt_j(M)$ is the set of prime ideals $P\in\Spec(R)$ such that $M_P$ cannot be generated by $j$ elements.    For details and the proofs of these statements we refer the reader  to the book of Eisenbud \cite[Section 20.2]{Ei}.

In this paper we are interested in the Fitting ideals $\Fitt_j(M)$, when $M$ is an ideal. The reason  for this arises from the Hilbert-Burch theorem. Let $R$ be a local ring,  and let $I\subset R$ be a perfect ideal of grade $2$. Then  the Hilbert-Burch theorem implies that $\Fitt_1(I)=I$. A similar statement holds for graded ideals in a graded $K$-algebra, where $K$ is a field. Considering this theorem  one may ask more generally for which ideals  $I$ we have $\Fitt_j(I)=I$ for some $j$. Dealing with this problem will be one of  our main concerns in this paper. 

In the first section of the paper we study the  relationship of an ideal to its Fitting ideals and prove in Theorem~\ref{contained} that for an ideal $I$ in a Noetherian ring which is generated by $m$ elements,  $I^{m-j}\subseteq \Fitt_j(I)$ for all  $1\le j\le m$. Equality holds,  if $I$ is generated by a regular sequence.  Moreover,  if $\grade I=j+1$, then  $\Fitt_{j}(I)\subseteq I$. As a result, in Corollary~\ref{Cor:RadFittGrade}  we obtain that  
$\sqrt{\Fitt_i(I)}=\sqrt{I}$  for  $i=1,\ldots,  \grade I-1$. Later in the paper, with a different argument, we show in Theorem~\ref{Thm:RadicalFitt} that  this statement remains  true if one replaces $\grade I$ by $\height I$, which in general is bigger than $\grade I$. Also from  Theorem~\ref{contained} we  obtain  in Corollary~\ref{equalm} a first simple example of an ideal $I$ for which $\Fitt_j(I)=I$ for some $j$. This is indeed the case when $(R,\mm)$ is local, $I=\mm$  and $j=\mu(\mm)-1$. Proposition~\ref{Chocklate} describes another case  under which a Fitting ideal of $I$ coincides with $I$. 
This case is for example given, when $R$ is a regular ring and $\mm$ is a maximal ideal of $R$ of height at least two.

The trace of a module $M$ is the ideal $\tr(M)=\sum_\varphi\varphi(M)$, where the sum is taken over all $\varphi\in\Hom_R(M,R)$. If $R$ is a domain and $M=I$ is an ideal,  then $\tr(I)=I^{-1}I$, see  \cite{HHS}. This allows us to compare $\tr(I)$ with $\Fitt_j(I)$. As a result we obtain in Proposition~\ref{twogen} that   $\tr(I)^{m-j}\subseteq \Fitt_j(I)$ for all $j$,  if $I$ is generated   by $m$ elements, and that   $\Fitt_1(I)=\tr(I)$,  if $I$  is generated by $2$ elements. 

In Section~\ref{two} we consider the trace of the canonical module $\omega_R$, and observe in Corollary~\ref{radomega} that the radicals of the ideals $\Fitt_1(\omega_R)$ and $\tr(\omega_R)$ coincide. If  $R$ is a local Cohen-Macaulay ring, which is generically a complete intersection, then  $\omega_R$ can be identified with an ideal of $R$, and one  may ask when it happens that $\Fitt_1(\omega_R)=\omega_R$. Trivially, this equality holds if $R$ is Gorenstein, and we  expect that equality holds if and only if $R$ is Gorenstein. So far we have not found any counterexample to this expectation, but also could not prove it yet. However, we show in Proposition~\ref{true} that if $\Fitt_{1}(\omega_R)\iso \omega_R$ and the Cohen-Macaulay type of $R$ is at most two, then $R$ is Gorenstein. 

Section~\ref{three} focusses on one of our main questions, namely, when  $\Fitt_1(I)=I$.   
In Theorem~\ref{david} a certain converse of the Hilbert-Burch theorem is presented. Indeed, it is shown that   if $R$ is  a local ring and  $I\subset R$ is  an ideal   with $\grade I\geq 2$ and $\Fitt_1(I)=I$, then $I$ is a perfect ideal of grade $2$.


Suppose that $I$ is an ideal with  $\grade I\geq j>2$. For such ideals one may presume, like in the case $j=2$, that   $I$ is a perfect ideal of grade  $j$, if $\Fitt_{j-1}(I)=I$. But it turns out that this is not the case. Indeed, in Theorem~\ref{noteasy} it is shown that if $S$ is a polynomial ring over a field, and  $I\subset S$ is a squarefree monomial ideal with $\grade I\geq j\ge 2$, then following conditions are equivalent:   (i) $\Fitt_{j-1}(I)=I$,  (ii)  $\Fitt_{j-1}(I)$ is squarefree, and (iii) If $j=2$, then $I$ is a perfect ideal with $\grade I=2$, and if $j>2$, then $I$ is generated by a regular sequence of length $j$. 	For the proof of this result combinatorial arguments are used.  But we expect that the equivalence of (i) and (iii)  is true for any ideal in $I$  in a local ring with  $\grade I\geq j\ge 2$. What we only can say in larger generality is phrased in Proposition~\ref{mirzaghasemi} which says that 
if    $I\subset R$  is a radical  ideal  of grade $\ge j\ge 2$ with  $\Fitt_{j-1}(I)=I$,   then $I$ is an  unmixed   ideal of grade $j$  and $R_P$ is regular for all $P\in \Ass I$.

For an ideal $I$ in a one-dimensional local ring, it may very well happen that $\Fitt_1(I)=I$. 
In Corollary~\ref{Brille} we show that if $(R,\mm)$ is a one-dimensional local domain  with infinite residue class field and multiplicity $e(R)=2$, then $\Fitt_1(I)=I$ if and only if $\tr(I)=I$. In Examples~\ref{demo}(a), this result is applied to numerical semigroup rings of multiplicity $2$, where the monomial ideals with $\Fitt_1(I)=I$ are classified.  In part (b) an  example of a monomial ideal $I$ in a numerical semigroup ring  of multiplicity 4 is given  which also satisfies $\Fitt_1(I)=I$, and which is not generated by $2$ elements. In view of this one may  ask  whether each one-dimensional local Cohen-Macaulay ring admits a proper ideal $I$ with $\Fitt_1(I)=I$.

In the last section of this paper we compare the radical of $\Fitt_j(I)$ with the radical of  $I$, when $I$ is an ideal in a Noetherian ring. As mentioned  before, these two radicals coincide  for $j=1,\ldots, \height I-1$. As a consequence we observe in Corollary~\ref{Cor:FittRadical} that in the same range for $j$, the $j$th Fitting ideal of $I$ and that of $\sqrt{I}$ coincide. 
Finally, in Theorem~\ref{edgeideal} we succeed to  compute $\sqrt{\Fitt_j(I)}$ for all $j$, when $I=I(G)$ is the edge ideal of a finite simple graph $G$.  The result shows that it in general it may be hard to find a nice description of $\sqrt{\Fitt_j(I)}$  for $j\geq \height I$.

\section{The relationship of an ideal to its Fitting ideals}

\medskip
In this section we compare the Fitting ideals of an ideal $I$ with powers of $I$, the radical and the trace of $I$. As the first result we have 

\begin{Theorem} 
	\label{contained}
	Let $R$ be a Noetherian ring and $I\subset R$ be an ideal generated by $m$ elements. Then the following hold: 
	
	\textup{(a)} $I^{m-j}\subseteq \Fitt_j(I)$ for all  $1\le j\le m$. Equality holds, if $I$ is generated by a regular sequence, or $R$ is local with maximal ideal $\mm$ and $I=\mm$.
	
	\textup{(b)} If $\grade I=j+1$, then $\Fitt_{j}(I)\subseteq I$.
\end{Theorem}  





\begin{proof}
	(a) By definition, $\Fitt_{j}(I)=I_{m-j}(A)$,  where  $A$ is a relation matrix of $I$. Let $\fb=f_1,\ldots,f_m$ be system of generators of $I$. 
	Consider the Koszul complex $K_\lpnt= (K_\lpnt(\fb;R),\partial_\lpnt)$ attached 
	to the sequence $\fb$. The complex  $K_\lpnt$  is a differential graded  algebra, whose algebra structure is  the exterior algebra of $K_1$. Let  $Z_1=\Ker \partial_1$  be the $R$-module  of  $1$-cycles of the Koszul complex. Then $Z_1^{m-j}\subset K_{m-j}$. After fixing a basis of $K_{m-j}$ each element of $Z_1^{m-j}$ can be uniquely written as a linear combination of the elements of this  basis. Then $I_{m-j}(A)$ is generated by the coefficients of this linear combinations as $z$ runs through a system of generators of $Z_1^{m-j}$.

	For $K_1$ we choose the basis $e_1,\ldots,e_m$ with $\partial_1(e_i)=f_i$ for $i=1,\ldots, m$. For any subset $A\subseteq [m]$, $A=\{a_1<a_2<\cdots < a_{\ell}\}$
	we set  $e_A=e_{a_1}\wedge  \cdots \wedge e_{a_{\ell}}$. Then the elements $e_A$ with $|A|={\ell}$ form a basis of $K_{\ell}$

	Let $1\leq j<m$, and let $f=f_{i_1}f_{i_2}\cdots f_{i_{m-j}}$  with  $1\leq i_1\leq i_2\leq \cdots \leq i_{m-j}<m$ be a generator of $I^{m-j}$. By induction on $m-j$  one  shows that there exist integers $1\leq k_1<\cdots < k_{m-j}\leq m$ such that $k_\ell\neq i_\ell$ for $\ell=1,\dots, {m-j}$. Let $z=z_1\wedge\cdots \wedge z_{m-j}$,  where $z_\ell=f_{i_\ell} e_{k_\ell}-f_{k_\ell }e_{i_\ell}$  is an element of $Z_1$. Then $z\in Z_1^{m-j}$, and the coefficient of $e_{k_1}\wedge\cdots\wedge e_{k_{m-j}}$ in $z$ is equal to $\pm f$. This proves that $I^{m-j}\subseteq  I_{m-j}(A)=\Fitt_j(I)$. 
	
Now, suppose that $f_1,\ldots,f_m$ is a regular sequence. Then $Z_1$ is generated by the $1$-boundaries $z_{i,j}=f_{i} e_j-f_{j }e_{i}$ with $1\leq i<j\leq m$. This implies that the coefficients of $Z_1^{m-j}$ belong to $I^{m-j}$. 
	
	
	Finally, assume that $R$ is local with maximal ideal $\mm$ and $I=\mm$. By the previous part $\mm^{m-j}\subseteq\Fitt_j(\mm)$ for all $1\le j\le m$. On the other hand, any nonzero entry of a  minimal  presentation matrix $A$ of $\mm$ belongs to $\mm$. Hence $\Fitt_j(\mm)=I_{m-j}(A)\subseteq\mm^{m-j}$ for all $1\le j\le m$, and this completes the proof.
	
	(b) Let ${\bf f}=f_1,\dots,f_m$ be a system of generators of $I$, and let as before $Z_1$ be the module of $1$-cycles of the Koszul complex $K_\lpnt(\fb;R)$. As observed in (a), the coefficients of the  elements   $z\in Z_1^{m-j}$ ( with respect the canonical basis of $K_{m-j}$)  generate $\Fitt_{j}(I)$. Since we assume that $\grade I=j+1$, it follows that  $H_{m-j}({\bf f};R)=0$, see \cite[Theorem 1.6.17(b)]{BH}. Therefore, any  $z\in Z_1^{m-j}$ is a boundary, since $Z_1^{m-j}\subseteq Z_{m-j}=B_{m-j}$, and so all coefficients of  $z$ belong to $I$. This shows that $\Fitt_{j}(I)\subseteq I$.
\end{proof}

\begin{Remark}
With  the assumptions and notation given in Theorem~\ref{contained}, part (b) of its proof shows that if  $Z_1^{m-j}=Z_{m-j}$ for $j=\grade I-1$, then $\Fitt_{j}(I)=I$.
\end{Remark}

\begin{Corollary}\label{Cor:RadFittGrade}
	Let $R$ be a Noetherian ring, and $I\subset R$ be an ideal with $\grade I=j$. Then
	$$
	\sqrt{\Fitt_i(I)}=\sqrt{I} \quad \text{for}\quad i=1,\ldots,  j-1.
	$$
\end{Corollary}
\begin{proof}
	By the previous theorem, we have $I^{m-1}\subseteq\Fitt_1(I)\subseteq \cdots \subseteq \Fitt_{j-1}(I)\subseteq  I$, where $m=\mu(I)$. Taking the radicals, the statement follows.
\end{proof}

In Theorem \ref{Thm:RadicalFitt} it is shown that  for the equality in Corollary \ref{Cor:RadFittGrade}, $\grade I$ can be replaced by $\height I$. 

\medskip
On the other hand, when $\grade I=j$, we cannot expect that $\sqrt{\Fitt_j(I)}= \sqrt{I}$. Indeed, consider the ideal $I=(x_1x_2,x_1x_3)\subset S=K[x_1,x_2,x_3]$ and $K$ a field. Then $\grade I=1$, $\sqrt{\Fitt_1(I)}=(x_2,x_3)$ and $\sqrt{I}=I$. More examples follow later.

\medskip  
 
When $I$ is the  maximal ideal in a local ring,  Theorem~\ref{contained} gives the answer to our general question of when  $\Fitt_j(I)=I$.

\begin{Corollary}
\label{equalm}
Let $R$ be a local ring with maximal ideal $\mm$. Then $\Fitt_j(\mm)=\mm$ if and only if $j=\mu(\mm)-1$.
\end{Corollary}


The next result describes another situation for which a Fitting ideal of an ideal $I$ coincides with $I$.

\begin{Proposition}\label{Chocklate}
Let $R$ be a Noetherian Cohen-Macaulay ring, and let $I\subset R$ be an ideal with $\mu(IR_P)=\grade I=j\geq 2$ for any $P\in V(I)$. Then $\Fitt_{j-1}(I)=I$.  
\end{Proposition}

\begin{proof}
The assumption $\mu(IR_P)=j$ together with \cite[Proposition 20.6]{Ei} implies that $\Fitt_{j}(I)=R$. Hence it follows that $\Fitt_{j-1}(I)=\Ann \wedge^jI$, see \cite[Exercises 20.9, 20.10]{Ei}. We set $J=\Ann \wedge^jI$. Then 
by Theorem~\ref{contained}, we have 
$J=\Fitt_{j-1}(I)\subseteq I$. We show that $J=I$. From the inequalities $$\grade I\leq \grade IR_P\leq\height IR_P\leq\mu(IR_P)=\grade I,$$
we obtain that $IR_P$ is a complete intersection for any $P\in V(I)$. Therefore, $\Ann \wedge^jIR_P=IR_P$. So we conclude that $JR_P=IR_P$ for any $P\in V(I)$. This together with the inclusion $J\subseteq I$ implies that $J=I$. 
\end{proof}

Using Proposition~\ref{Chocklate} and under the additional assumption that $R$ is a regular ring, Corollary~\ref{equalm} can be extended as follows.

\begin{Corollary}
Let $R$ be a regular ring (not necessarily local), and let $\mm$ be maximal ideal of $R$ such that $\dim R_{\mm}=d\geq 2$. Then
$\Fitt_{d-1}(\mm)=\mm$.
\end{Corollary}

 The next result shows that if $I$ is a $2$-generated ideal in a domain, then $\Fitt_1(I)$ is the trace of $I$. 
 
 
 \begin{Proposition}	
 	\label{twogen}
 	Let  $R$ be a domain, and let $I\subset R$ be an ideal.  Then  we have 
 	\begin{enumerate}
 		\item[ (a)] If $I$ is generated   by $m$ elements, then $\tr(I)^{m-j}\subseteq \Fitt_j(I)$ for all $j$.
 		\item[(b)]  If $I$  is generated by $2$ elements, then $\Fitt_1(I)=\tr(I)$. 
 	\end{enumerate}
 \end{Proposition}
 
 \begin{proof}
 	(a) Let $h\in I^{-1}$. Then by Theorem~\ref{contained},  $(hI)^{m-j}\subseteq  \Fitt_j(hI)=\Fitt_j(I)$. So $(hf)^{m-j}\in \Fitt_j(I)$, for all $f\in I$ and $h\in I^{-1}$. This implies the desired inclusion.
 	
 	
 	(b) By (a), it is enough to show that  $\Fitt_1(I)\subseteq \tr(I)$. Let  $I=(f_1,f_2)$. Let $g_1f_1+g_2f_2=0$,  and let $h=g_1/f_2$. Then $hf_2=g_1$ and $hf_2f_1=g_1f_1=-g_2f_2$. Therefore, $hf_1=-g_2$. This implies that $h\in I^{-1}$. Conversely, if $h\in I^{-1}$, then $g_1f_1+g_2f_2=0$ with $g_1=hf_2$ and $g_2=-hf_1$.  Thus,  the desired conclusion follows. 
 \end{proof}

\section{Fitting ideals of the canonical module}
\label{two}
\medskip
Let $(R,\mm)$ be a Cohen-Macaulay local ring with  canonical module $\omega_R$.  The Cohen-Macaulay type of a Cohen-Macaulay module $M$ will be denoted by $r(M)$. We have 

\begin{Lemma}
\label{type}
Let $j\geq 1$ be  an integer. Then the set of prime ideals $P$ of $R$ for which $r(R_P)\geq j+1$ is a closed set in $\Spec(R)$. 
\end{Lemma}

\begin{proof}
We note that $\mu(\omega_R) =r(R_P)$. Hence the assertion follows from the fact that $V(\Fitt_j(\omega_R))$ is the set  of prime ideals $P$ for which $(\omega_R)_P=\omega_{R_P}$ cannot be generated by $j$ elements. 
\end{proof}

\begin{Corollary}
\label{radomega}
Let $R$ be a Cohen-Macaulay local ring. Then $\sqrt{\Fitt_1(\omega_R)}=\sqrt{\tr(\omega_R)}.$
\end{Corollary}

\begin{proof}
For $P\in \Spec(R)$, $R_P$ is Gorenstein if and only if $\omega_{R_P}$ is generated by one element. Thus, $\Fitt_1(\omega_R)$ describes the non-Gorenstein locus of $R$. Since the non-Gorenstein locus of $R$ is also given by $\tr(\omega_R)$, as noted in \cite{HHS}, the statement follows.
\end{proof}

In the next two results we assume that $(R,\mm)$                  is a local Cohen--Macaulay ring which is generically Gorenstein. We furthermore assume that $R$ admits a canonical module $\omega_R$. By \cite[Proposition 3.3.18]{BH},  $\omega_R$ can be identified with a  Cohen-Macaulay ideal of  grade  $1$. 
      Assume that $\Fitt_j(\omega_R)\iso \omega_R$ for some $j$. Then  there exists a non-zerodivisor $t\in R$ such that $\Fitt_j(\omega_R)  =t\omega_R$.  Then $\Fitt_j(t\omega_R)=\Fitt_j(\omega_R)=t\omega_R$. Since the ideal $t\omega_R$  is also   a canonical ideal of $R$, the condition   $\Fitt_j(\omega_R)\iso \omega_R$ can always be replaced by 
   $\Fitt_j(\omega_R)= \omega_R$    for a suitable choice of  $\omega_R$.

 \begin{Corollary}
 \label{tromega} 
 Assume that $\Fitt_1(\omega_R)=\omega_R$. Then $R_P$ is Gorenstein if and only if $\omega_R\not\subseteq P$.  Moreover,   if $R$ is not Gorenstein, then  the non-Gorenstein locus of $R$ is of dimension $\dim R-1$.
 \end{Corollary}
 
 \begin{proof}
Let $P\in\Spec(R)$. Our assumptions and  Lemma~\ref{type} imply that  $R_P$  is not Gorenstein if and only if $\omega_R\subseteq P$. 

If $R$ is not Gorenstein, then it follows  from our assumption that $\omega_R$ is a proper ideal of $R$. The ideal $\omega_R$ is of height one, since it is a Cohen-Macaulay module. This together with the first part of the proof completes the proof. 
\end{proof}
 
  Assume that $R$ is not Gorenstein, and let $P$ be a minimal prime ideal of the canonical ideal $\omega_R$. The proof of Corollary ~\ref{tromega} shows that if $\Fitt_1(\omega_R)=\omega_R$,  then $R_P$ is a one-dimensional   Cohen-Macaulay ring  but  not Gorenstein.
So far we could not find a one-dimensional non-Gorenstein ring with  canonical ideal $\omega_R$ for which $\Fitt_1(\omega_R)=\omega_R$.  
However, we have

 \begin{Proposition}
 \label{true}
Let $R$ be a local Cohen-Macaulay domain with the canonical module $\omega_R$. Suppose that $\Fitt_{1}(\omega_R)=\omega_R$ and  the Cohen-Macaulay type of $R$ is at most $2$. Then $R$ is Gorenstein. 
 \end{Proposition}
 
 \begin{proof}
Suppose that $R$ is not Gorenstein. Then $\mu(\omega_R)=2$ and it follows from Proposition~\ref{twogen} and our hypothesis that $\omega_R=\omega_R^{-1}\omega_R$. This implies that 
\[
R=\omega_R:\omega_R=\omega_R:(\omega_R^{-1}\omega_R)=(\omega_R:\omega_R):\omega_R^{-1}=R:\omega_R^{-1}=(\omega_R^{-1})^{-1}.
\]
Therefore, $R=((\omega_R^{-1})^{-1})^{-1}=\omega_R^{-1}$. 

Since $\Hom_R(R/\omega_R, R)=0$, the exact sequence  $0\to \omega_R\to R\to R/\omega_R\to 0$, induces the  exact sequence 
\[
0\to R\to \omega_R^{-1} \to \Ext^1_R(R/\omega_R, R)\to 0.
\]
It follows that $\Ext^1_R(R/\omega_R, R)= 0$.  Since $R$ is not Gorenstein, $R/\omega_R\neq 0$, and since $\grade \omega_R =1$, the theorem of Rees (cf.~\cite[Theorem 1.2.5]{BH}) implies that $\Ext^1_R(R/\omega_R, R)\neq 0$, a contradiction.
\end{proof}
	
\section{When is $\Fitt_j(I)=I$?}
\label{three}

Let $R$ be a Noetherian local ring or a finitely generated graded $K$-algebra. The Hilbert-Burch theorem implies that if $I\subset R$ is a perfect ideal of grade $2$, then $\Fitt_1(I)=I$. In this section we are interested in a converse of this theorem, and ask more generally that if $\grade I\geq j$ and $\Fitt_{j-1}(I)=I$, then what can be said about $I$. The following theorem gives an answer for $j=2$.   

\begin{Theorem}
\label{david}
Let $R$ be a local ring, and let $I\subset R$ be an ideal. If $\grade I\geq 2$ and $\Fitt_1(I)=I$,
then $I$ is a perfect ideal of grade $2$.
\end{Theorem}

\begin{proof}
We show that the projective dimension of $I$ is $1$.
Let
$$\cdots\longrightarrow R^m\stackrel{\varphi_2}\longrightarrow R^n\stackrel{\varphi_1}\longrightarrow R\to R/I\longrightarrow 0$$
be a minimal free resolution of $R/I$. Since $\grade I\geq 2$, the first two
steps of the dual of this resolution is a resolution of $\Coker(\varphi_2^*)$. 
 By \cite[Theorem 3.1]{BE2},
there are maps $a_0, a_1$ and  $a_2$ making the following diagrams commute:

\begin{center} \begin{tikzcd}[column sep=small]
	& R \arrow[dr,"a_0"] & \\
	R^n \arrow{rr}{\varphi_1} \arrow{ur}{a_1^*}& & R 
\end{tikzcd}
\end{center}

\begin{center}\begin{tikzcd}[column sep=small]
	& R \arrow[dr,"a_1"] & \\
	\wedge^{n-1} R^m \arrow{rr}{\wedge^{n-1}\varphi_2} \arrow{ur}{a_2^*}& & \wedge^{n-1} R^n. 
\end{tikzcd}
\end{center}
For a map $\varphi$ of free modules, let $I_j(\varphi)$ be the ideal of $j$-minors of $\varphi$. Since $\grade I\geq 2$, the map $a_0$ must be a unit, so the ideal $I_1(a_1)$ 
is equal to $I$.  Since, by assumption, $I_{n-1}(\varphi_2)=I=I_1(a_1)$, it follows that $a_2^*$
is surjective.
Thus one of the usual basis elements of $\wedge^{n-1} R^m$ must map by $a_2^*$ to a unit, and it follows
that if $\varphi'_2$ is the restriction to an appropriate summand of rank $n-1$ in $R^m$, then $I_{n-1}(\varphi'_2)$
has grade $2$. Thus by \cite[Corollary 1]{BE1}, 
$$0\longrightarrow R^{n-1}\stackrel{\varphi'_2}\longrightarrow R^n\stackrel{\varphi_1}\longrightarrow R\to R/I\longrightarrow 0$$
is a resolution.
\end{proof}
	
\medskip

Now we consider  radical ideals $I$ of grade $\geq j\geq 2$ which satisfy the condition   $\Fitt_{j-1}(I)=I$.
Recall that an ideal $I$ is called {\em unmixed}, if all the associated prime ideals of $I$ have the same height, equal to $\height I$.

\begin{Proposition}\label{mirzaghasemi}
  Let  $j\ge 2$ be an integer,  and let  $I\subset R$ a radical  ideal  of grade $\ge j$. If $\Fitt_{j-1}(I)=I$,   then $I$ is an  unmixed   ideal of grade $j$  and $R_P$ is a regular ring for all $P\in \Ass I$.
  \end{Proposition}
  
 \begin{proof}
 Let $P\in \Ass I$. Then $PR_P=IR_P$, since $I$ is reduced.  Let $m=\mu(PR_P)$. Theorem~\ref{contained} and our assumptions imply that 
\[
PR_P=IR_P= \Fitt_{j-1}(I)R_P=\Fitt_{j-1}(IR_P)=\Fitt_{j-1}(PR_P)=P^{m-j+1}R_P. 
\]
It follows that $\mu(PR_P)=j$. Krull's generalized principal ideal theorem implies that  $\height P= \height PR_P\leq j$. On the other hand, $j\leq \grade I\leq \height P$. Therefore,
 $\grade  I=j$ and $\height P=j$ for all $P\in \Ass I$. This shows that $I$ is an unmixed ideal of grade $j$ .
 
We also have seen that $\mu(PR_P)=\height PR_P=\dim R_P$ for all $P\in \Ass I$.  It follows  that $R_P$ is regular for all $P\in \Ass I$. 
\end{proof}

In the following theorem we answer the main question of this section for squarefree monomial ideals.
Before stating this result we need to recall some concepts and notation on graphs. 

For a finite simple graph $G$, we denote  the vertex set and the edge set of $G$ by $V(G)$ and $E(G)$, respectively. The \textit{edge ideal} of $G$ is the squarefree monomial ideal in the polynomial ring $S=K[x_i:\ i\in V(G)]$ generated by all monomials $x_ix_j$ such that $\{i,j\}\in E(G)$.
The \textit{complementary graph} of $G$ is the graph $G^c$ on vertex set $V(G)$ whose edges are the non-edges of $G$. A cycle $C$ in $G$ is called an {\em induced cycle} if no non-adjacent vertices of $C$ form an edge in $G$. The graph $G$ is called {\em chordal} if it has no induced cycles of length bigger than $3$.

 \begin{Theorem}
 	\label{noteasy}	
 	Let  $I\subset S$ be a squarefree monomial ideal with $\grade I\geq j\ge 2$. The following conditions are equivalent:   
 	\begin{enumerate}
 	\item [(i)] $\Fitt_{j-1}(I)=I$.
 \item [(ii)]  $\Fitt_{j-1}(I)$ is squarefree. 
 \item [(iii)] If $j=2$, then $I$ is  a perfect ideal with $\grade I=2$, and if $j>2$, then $I$ is generated by a regular sequence of length $j$. 	
 	\end{enumerate}

 \end{Theorem}
 \begin{proof}
 	
(i)\implies (ii) is obvious. 

(ii)\implies (i): Let $\Fitt_{j-1}(I)$ be squarefree. Then by Corollary~\ref{Cor:RadFittGrade}, we have $$\Fitt_{j-1}(I)=\sqrt{\Fitt_{j-1}(I)}=\sqrt{I}=I.$$

(i)\implies (iii): It follows from Proposition~\ref{mirzaghasemi} that $I$ is an unmixed ideal of grade $j$.
First assume that $j=2$. Then $I$ is unmixed of height $2$.		
It is enough to show that $S/I$ is a Cohen-Macaulay ring.  Let $G$ be the graph on $[n]$ such that $\{i,j\}\in E(G)$ if and only if $(x_i,x_j)$ is a minimal prime ideal of $I$. Then $I=\bigcap_{\{i,j\}\in E(G)} (x_i,x_j)$.
 Since $I^{\vee}=I(G)$, by Eagon-Reiner \cite[Theorem 8.1.9]{JT}, $S/I$ is a Cohen-Macaulay ring if and only if $I(G)$ has a linear resolution. Thus by \cite{Froberg88}, (see, also, \cite[Theorem 9.2.3]{JT}), we need to show that $G^c$ is a chordal graph. By contradiction assume that $G^c$ has an induced cycle $C_k$ with $k\ge 4$. Then $C_k^c$ is an induced subgraph of $G$. 
 	Consider the prime ideal $P=(x_i:\ i\in V(C_k))$ in $S$ and set $J=\bigcap_{\{i,j\}\in E(C_k^c)} (x_i,x_j)$. Then $\Fitt_1(JS_P)=\Fitt_1(IS_P)=IS_P=JS_P$.
 	The minimal monomial generators of $J$ are of the form $x_A$, where $A$ is a minimal vertex cover of $C_k^c$. One can see that each minimal vertex cover $A$ of $C_k^c$ is of the form $A=V(C_k)\setminus e$, where $e\in E(C_k)$. Therefore $JS_P$ is  generated in degree $k-2$ and $\mu(JS_P)=k$. On the other hand, $\Fitt_1(JS_P)$ is the ideal which is generated by $(k-1)$-minors of a relation matrix of $JS_P$. So each monomial generator of $\Fitt_1(JS_P)$ is of degree at least $k-1$. This contradicts to the equality $\Fitt_1(JS_P)=JS_P$. Hence $G^c$ is chordal, as claimed.
 	
Let $j>2$. Let $\mathcal{G}(I)=\{u_1,\ldots,u_m\}$ be the set of minimal monomial generators of $I$. Without loss of generality we may assume that $\bigcup_{i=1}^m \supp(u_i)=[n]$. Let $k_I$ be the number of minimal prime ideals of $I$. By induction on $k_I$ we show that if $\Fitt_{j-1}(I)=I$ and  $\grade I=j$, then $u_1,\ldots,u_m$ form a regular sequence. If $k_I=1$, then $I$ is a prime ideal generated by variables and there is nothing to prove. So we may assume that $k_I\ge 2$. We claim that  $m=j$. By contradiction assume that $m>j$. Any monomial generator of $\Fitt_{j-1}(I)$ has degree at least $m-j+1\ge 2$. So $\deg(u_i)\geq 2$ for all $i$. For each $1\le i\le n$, let $J_i$ be the monomial localization of $I$ at $x_i$, that is the monomial ideal obtained from $I$ 
by substituting the variable $x_i$ by $1$.  Then $J_i$ is a monomial ideal of grade $j$ with $k_{J_i}<k_I$ and $\Fitt_{j-1}(J_i)=J_i$. So by induction hypothesis we may assume that $J_i$ is generated by a regular sequence of monomials of length $j$. This implies that for distinct $i$ and $j$ if $\gcd(u_i,u_j)\neq 1$, then it has degree at most one. Hence any entry of a relation matrix $A$ of $I$ has degree at least $d-1$, where $d=\min\{\deg(u_i):\ 1\le i\le m\}$. Therefore any $(m-j+1)$-minor of $A$ has degree at least $(m-j+1)(d-1)$. Since $\Fitt_{j-1}(I)=I$, we obtain $(m-j+1)(d-1)\le d$. Since $d\ge 2$, we get $m-j+1=2$ and $d=2$. So $m=j+1$.
 	Without loss of generality let $u_1=x_1x_2$ be a monomial of degree $2$ in $I$. Since $x_1x_2\in \Fitt_{j-1}(I)$, there exists $u_2,u_3\in \mathcal{G}(I)$ such that $x_1|u_2$ and $x_2|u_3$. The equality $m=j+1$ together with $\height I=j$ implies that any variable $x_i$ appears in at most two of the $u_i$'s.  Note that $m\ge 4$. Since $m=j+1$ and $(u_1,u_2,u_3)\subset (x_1,x_2)$
 	we conclude that $u_4,\ldots,u_m$ form a regular sequence. Therefore, from $u_4\in \Fitt_{j-1}(I)$, we have $u_4=(u_2/x_1)(u_3/x_2)$. Since $\gcd(u_2,u_4)$ has degree at most one, we conclude that $u_2=x_1x_r$ for some $r$.
 	Similarly, $u_3=x_2x_s$ for some $s$. Hence $u_4=x_rx_s$. If $m>4$, then $u_5=x_rx_s$ as well, which is not possible. Hence $m=4$ and we may write $I=(x_1x_2,x_2x_3,x_1x_4,x_3x_4)$. Hence $\height I=2$, a contradiction.  
 	
(iii)\implies (i) follows from the Hilbert-Burch Theorem together with Theorem~\ref{contained}(a).	     
 \end{proof}
For an ideal $I$ in a one-dimensional local ring, it may very well happen that $\Fitt_1(I)=I$, as we will see below. 

An ideal $I$ in a ring $R$ is called a {\em trace ideal}, if $I=\tr(M)$ for some $R$-module $M$. By~\cite[Proposition 2.8]{Lindo}, $I$ is a trace ideal if and only if $I=\tr(I)$.  
 
\begin{Corollary}\label{Brille}
 	Let $(R,\mm)$ be a one-dimensional local domain  with infinite residue class field and multiplicity $e(R)=2$. Then $\Fitt_1(I)=I$ if and only if $I$ is a trace ideal. 
 \end{Corollary}
 
 \begin{proof}
 	Since the residue class field is infinite, there exists an element $x\in R$ such that $e(R)=\ell(R/(x))$. Let $I\subset R$ be an arbitrary ideal. Then $\mu(I)\leq \ell(I/xI)=\ell(R/(x))=e(R)=2$, see for example~\cite[Corollary 4.7.11]{BH}. So by Proposition~\ref{twogen} we have $\Fitt_1(I)=\tr(I)$ for all ideals $I$ in $R$. Now, by~\cite[Proposition 2.8]{Lindo} the assertion follows.
 \end{proof}
 
 \begin{Examples} \label{demo}
 	{\em  (a) Let $R=K\llbracket t^2,t^{2k+1}\rrbracket$ be a numerical semigroup ring, and let $I\subset R$ be a monomial ideal which is not principal. Then there exists  an   integer   $i=1,\ldots,k$ such that  $I$ is isomorphic to the monomial  fractionary ideal $J=(1, t^{2i-1})$. The integer $i$ is uniquely determined by $I$.   By Proposition~\ref{twogen},  $\Fitt_1(I)=I^{-1}I=J^{-1}J$.  We have $J^{-1}=(t^{2(k-i+1)}, t^{2k+1})$. Therefore, $\Fitt_1(I)=(t^{2(k-i+1)}, t^{2k+1})(1,t^{2i-1})=(t^{2(k-i+1)}, t^{2k+1})$. 
 By Corollary~\ref{Brille}, $\Fitt_1(I)=I$ if and only if $I=(t^{2(k-i+1)}, t^{2k+1})$ for some $i\in [k]$.

 		(b)  Let $R=K\llbracket t^4,t^5\rrbracket$, and let $I=(t^{12},t^{13},t^{14},t^{15})$. Then $\Fitt_1(I)=I$.
 	}
 \end{Examples}

Even when the multiplicity of a one-dimensional local domain $R$ is not $2$, there may exist a proper ideal $I\subset R$ with  $\Fitt_1(I)=I$, as the above example shows. We expect that for any one-dimensional local Cohen-Macaulay ring there always exists an ideal $I$ with $\Fitt_1(I)=I$.



\section{On the radical of Fitting ideals}

In this section we compare  the ideal $I$ with the radical of $\Fitt_j(I)$.
\begin{Theorem}\label{Thm:RadicalFitt}
	Let $R$ be a Noetherian ring and $I\subset R$ be an ideal of $\height I=h$. Then for any $1\le j\le h-1$ we have
	$$
	\sqrt{\Fitt_j(I)}=\sqrt{I}.
	$$
\end{Theorem}
\begin{proof}
Fix an integer $1\le j\le h-1$. It follows from Theorem \ref{contained} that $\sqrt{I}\subseteq\sqrt{\Fitt_j(I)}$. Since $\sqrt{I}=\bigcap_{P\in V(I)}P$, we only need to show that $\Fitt_j(I)\subseteq P$ for any prime ideal $P$ containing $I$. Let $P$ be such a prime ideal.  Since $j+1\le h=\height I\le\height IR_P\le \mu(IR_P)$, it follows from \cite[Proposition 20.6]{Ei} that $\Fitt_j(IR_P)\subseteq PR_P$. Hence 
$\Fitt_j(I)\subseteq \Fitt_j(I)R_P\cap R\subseteq PR_P\cap R=P$, as desired.
\end{proof}

\begin{Corollary}\label{Cor:FittRadical}
	Let $R$ be a Noetherian ring and let $I\subset R$ be an ideal of $\height I=h$. Then for any $1\le j\le h-1$ we have
	$$
	\sqrt{\Fitt_j(I)}=\sqrt{\Fitt_j(\sqrt{I})}.
	$$
\end{Corollary}
\begin{proof}
	Since $\height I=\height \sqrt{I}$ and $\sqrt{\sqrt{I}}=\sqrt{I}$, the claim follows from the previous theorem.
\end{proof}

Next, we determine radical of Fitting ideals of $I$, when $I$ is the edge ideal of a graph.  Let $G$ be a finite simple graph with the vertex set $V(G)=[n]$. 
 The \textit{open neighbourhood} of $i\in V(G)$ is the set $N(i)=\{j:\{i,j\}\in E(G)\}$. Whereas, the \textit{closed neighbourhood} of $i$ is the set $N[i]=N(i)\cup\{i\}$.
For a subset $A$ of $V(G)$, we denote by $G\setminus A$, the subgraph of $G$ with the vertices of $A$ and their incident edges deleted.
An edge $e$ of $G$ is called a \emph{neighbour} of $i$ if $i\notin e$ and $i$ is adjacent to an endpoint of $e$. The \emph{edge neighbourhood} of $i$ is the set of all edges in $G$ which are neighbours of $i$ and is denoted by $E(i)$.  		
Let $F=\{i_1,\ldots,i_k\}\subseteq V(G)$. We call a subset $A=\{e_1,\ldots,e_t\}$ of $E(G)$ an \emph{admissible cover} of $F$ of size $t$ if:
\begin{enumerate}
	\item[\textup{(a)}] $A=\bigsqcup_{s=1}^k A_s$, where $\emptyset\neq A_s\subseteq E(i_s)$ for all $s$.
	\item[\textup{(b)}] $\{i_s\}\cup e\neq \{i_t\}\cup e'$ for any distinct integers $s,t\in [k]$, $e\in A_s$ and $e'\in A_t$. 
\end{enumerate}




The minimum cardinality of a vertex cover of $G$ is denoted by $c_G$.

\begin{Theorem}\label{edgeideal}
	Let $G$ be a graph with $m$ edges. Then
		
		$$
	\sqrt{\Fitt_j(I(G))}=\begin{cases}
		I(G)&\text{if}\ j<c_G,\\
		 I(G)+J&\text{if}\ j\ge c_G,
	\end{cases}
	$$
	where $J$ is the squarefree monomial ideal minimally generated by those monomials $x_{i_1}\cdots x_{i_k}$ for which 
	$F=\{i_1,\ldots,i_k\}$ is an independent set of $G$, $F$ has an admissible cover of size $m-j$ and no proper subset of $F$ has such an admissible cover.
\end{Theorem}
\begin{proof}
If $j<c_G$, the result follows from Theorem~\ref{Thm:RadicalFitt}, since $\height I(G)=c_G$.	Let $j\ge c_G$, let $V(G)=[n]$, $E(G)=\{e_1,\ldots,e_m\}$ and $I=I(G)$. For any edge $e_t=\{i,j\}$, we set $f_{e_t}=x_ix_j$.

Note that any Fitting ideal of a monomial ideal is a monomial ideal. First we show that for any minimal generator $x_{i_1}\cdots x_{i_k}\in \sqrt{\Fitt_j(I)}\setminus I$ of  $\sqrt{\Fitt_j(I)}$, $F=\{i_1,\ldots,i_k\}$ is an independent set of $G$ and $F$ has an admissible cover of size $m-j$. The fact that $F$ is an independent set of $G$ follows from the assumption that $x_{i_1}\cdots x_{i_k}\notin I$. Moreover, since $x_{i_1}\cdots x_{i_k}$ is a minimal generator of $\sqrt{\Fitt_j(I)}$, we have $x_{i_1}^{r_1}\cdots x_{i_k}^{r_k}\in \Fitt_j(I)$ for some positive integers $r_1,\ldots,r_k$ with $r_1+\cdots+r_k=m-j$. This means that there is a $(m-j)$-minor $x_{i_1}^{r_1}\cdots x_{i_k}^{r_k}$ of a relation matrix $B$ of $I$ whose columns correspond to some relations of $I$. Up to relabeling, there exists a subset $A=\{e_1,\ldots, e_{m-j}\}$ of $E(G)$ with $A=\bigsqcup_{s=1}^k A_s$ such that for any $1\le s\le k$, $|A_s|=r_s$ and for any $e_{\ell}\in A_s$,  $x_{i_s}f_{e_{\ell}}=x_pf_{e'}$ for some $p\in V(G)$ with $p\neq i_s$ and some $e'\in E(G)$.  Hence if $e_{\ell}=\{p,q\}$, then $\{i_s,q\}=e'\in E(G)$. Therefore $e_{\ell}\in E(i_s)$ and hence $A_s\subseteq  E(i_s)$.

	To show property (b) of $A$ to be an admissible cover of $F$, by contradiction assume that there exist  distinct integers $s,t\in [k]$, $e\in A_s$ and $e'\in A_t$ such that  $\{i_s\}\cup e=\{i_t\}\cup e'$. Note that  $\{i_s\}\cup e=\{p\}\cup e''$ for some $p\in [n]$ and $e''\in A$ and
	$\{i_t\}\cup e'=\{q\}\cup e_0$ for some $q\in [n]$ and $e_0\in A$.  Thus $e''=e_0$ and $p=q$. This means that $e''=\{i_s,i_t\}$ and hence $x_{i_1}\cdots x_{i_k}\in I$, which is a contradiction.    
	This shows that $F$ has the admissible cover $A$ of size $m-t$.

	Next we show that if $F=\{i_1,\ldots,i_k\}$ is an independent set of $G$ and   $A=\bigsqcup_{s=1}^kA_s=\{e_1,\ldots, e_{m-j}\}$ is an admissible cover of $F$, then $x_{i_1}\cdots x_{i_k}\in \sqrt{\Fitt_j(I)}\setminus I$. To prove this, first by relabeling the $e_i$'s we may assume that if $s<t$, then for any $e_i\in A_s$ and $e_{\ell}\in A_t$ one has $i<\ell$.  
	For any $e_t\in A_s$, there exists a relation $x_{i_s}f_{e_t}=x_pf_{e_r}$ for some $p\neq i_s$ and some $e_r\in E(G)$. Let $B_t$ be the $m\times 1$-matrix corresponding to this relation, i.e., whose $t$-th entry is $x_{i_s}$, its $r$-th entry is $-x_p$ and its other entries are zero. Let $B$ be the $m\times (m-j)$-matrix whose $i$-th column is $B_i$, and let $B'$ be the $(m-j)$-minor of $B$ obtained by deleting the last $j$ rows of $B$. Then, we see that $B'$ is the determinant of the following matrix:
	
	$$
	\begin{array}{l}
		A_1\begin{cases}
			f_{e_1}\\
			\\
			\\
		\end{cases}
		\\[3pt] \,\,\,\,\,\,\,\,\,\,\,\,\,\,\,\vdots \\[3pt]
		A_k\begin{cases}
			\\
			\\
			f_{e_{m-j}}\\
		\end{cases}
	\end{array}
	\!\!\!\!\!\!\!\!\!\!\! \left[
	\begin{array}{ccccccc}
		x_{i_1}&&&&&&\\
		&\ddots&&&&&\\
		&&x_{i_1}&&&&\\[5pt]
		&&&\,\,\ddots\,\,&&&\\[5pt]
		&&&&x_{i_k}&&\\
		&&&&&\ddots&\\
		&&&&&&x_{i_k}
	\end{array}
	\right].
	$$\medskip
	
	If $\det(B')\neq 0$, then it is equal to $x_{i_1}^{r_1}\cdots x_{i_k}^{r_k}$, where $r_i=|A_i|$ for all $i$.  For any nonzero entry $x_t$ of $B$ with $t\notin F$ replacing 
	$x_t$ by zero in $B'$ does not change $\det(B')$. Indeed, if $\det(B')\neq 0$,  since $x_t$ does not divide $\det(B')$ this is the case, and if $\det(B')=0$ this is obviously true, as well. Hence without loss of generality we may assume that each column of $B'$ has at most two nonzero entries which belong to $\{x_{i_1},\ldots,x_{i_k}\}$. 
	We claim that the columns of $B'$ are linearly independent. If this is not the case, then one can find two columns $B'_i$ and $B'_j$ of $B'$ whose $r$-th rows are equal up to sign to $x_{i_{\ell}}$ for some $r$. Then from the shape of $B'$ one can see that at least one of $B'_i$ and $B'_j$ has another nonzero entry, say $x_{i_t}$. This means that we have the relation $x_{i_t}f_{e_d}=x_{i_{\ell}}f_{e_s}$
	for some $\ell,s\in [m-j]$. This contradicts to property (b) of being an admissible cover. Hence, the columns of $B'$ are linearly independent and so $\det(B')\ne0$ is the required monomial in $\Fitt_j(I)$.
Therefore $x_{i_1}\cdots x_{i_k}\in\sqrt{\Fitt_j(I)}\setminus I$. 
Now, in addition assume that no proper subset of $F$ has an admissible cover of size $m-j$. Then $x_{i_1}\cdots x_{i_k}$ is a minimal generator of $\sqrt{\Fitt_j(I)}$. Otherwise, there exists $F'\subsetneq F$ such that $\prod_{t\in F'} x_t$ is a minimal generator of $\sqrt{\Fitt_j(I)}$. Then it follows from the first part of the proof that $F'$ has an admissible cover of size $m-j$, a contradiction. 

We complete the proof by showing that if $x_{i_1}\cdots x_{i_k}\in \sqrt{\Fitt_j(I)}\setminus I$ is a  minimal generator of  $\sqrt{\Fitt_j(I)}$  and $F=\{i_1,\ldots,i_k\}$, then no proper subset of $F$ has an admissible cover of size $m-j$. On the contrary suppose that $F'\subsetneq F$ is a set which has an admissible cover of size $m-j$. Then $F'$ is an independent set of $G$ as well. As was shown above these imply that $\prod_{t\in F'} x_t\in \sqrt{\Fitt_j(I)}\setminus I$, which contradicts to the minimality of $x_{i_1}\cdots x_{i_k}$.    
\end{proof}		

As a corollary, we characterize when the radical of Fitting ideal of an edge ideal $I(G)$ is the graded maximal ideal, in terms of the combinatorics of $G$.

\begin{Corollary}\label{maximal}
Let $G$ be a graph on $[n]$ and let $m=|E(G)|$. Then $x_i\in \sqrt{\Fitt_j(I)}$ if and only if $|E(i)|\ge m-j$. In particular, $\sqrt{\Fitt_j(I)}=(x_1,\ldots,x_n)$ if and only if $m-\min\{|E(i)|:\ i\in [n]\}\le j<m$.
\end{Corollary}

Let $d$ be a positive integer. A graph $G$ is called \textit{$d$-regular} if $|N(i)|=d$ for all $i\in V(G)$. For instance, a complete graph on $n$ vertices is $(n-1)$-regular, whereas any cycle is 2-regular.

\begin{Remark}\label{Prop:Fitt(I(G))=m}
{\em	The second statement of Corollary~\ref{maximal} can be rephrased as follows, as well.  We have $\sqrt{\Fitt_j(I)}=(x_1,\dots,x_n)$ if and only if 
	$$
	\max\{|N(i)|+|E(G_i)|:\ i\in V(G)\}\le j<m,
	$$
	where $G_i=G\setminus N[i]$.
	In particular, when $G$ is a $d$-regular graph, then $\sqrt{\Fitt_j(I)}=(x_1,\dots,x_n)$ if and only if $j\ge d+|E(G_i)|$ for each $i$.
}
\end{Remark}

\begin{Example}
	\rm Let $G=K_n$ be the complete graph on $[n]$ and $I=I(G)$. Then,
	$$
	\sqrt{\Fitt_j(I)}=\begin{cases}
		I&\text{if}\ j\le n-1,\\
		(x_1,\dots,x_n)&\text{if}\ n-1\le j<{n\choose 2},\\
		S& \text{otherwise}.
	\end{cases}
	$$
\end{Example}


\end{document}